\documentclass[11pt]{article}
\usepackage{amsfonts}
\usepackage{amssymb}
\usepackage{amsthm}
\usepackage{amsmath}
\usepackage{graphicx}
\usepackage{empheq}
\usepackage{indentfirst}
\usepackage{cite}
\usepackage{mathrsfs}
\usepackage{graphics}
\newtheorem{theorem}{\textbf{Theorem}}[section]
\newtheorem{lemma}{\textbf{Lemma}}[section]
\newtheorem{proposition}{\textbf{Proposition}}[section]
\newtheorem{corollary}{\textbf{Corollary}}[section]
\newtheorem{remark}{\textbf{Remark}}[section]
\newtheorem{definition}{\textbf{Definition}}[section]

\allowdisplaybreaks[4]

\def\be{\begin{equation}}
\def\ee{\end{equation}}
\def\bea{\begin{eqnarray}}
\def\eea{\end{eqnarray}}
\def\bt{\begin{theorem}}
\def\et{\end{theorem}}
\def\bl{\begin{lemma}}
\def\el{\end{lemma}}
\def\bc{\begin{corollary}}
\def\ec{\end{corollary}}
\def\bd{\begin{definition}}
\def\ed{\end{definition}}

\def\non{\nonumber }

\def\F{\mathsf{F}}
\def\E{\mathsf{E}}
\def\S{\mathsf{S}}
\def\i{\mathrm{i}}
\def\u{\mathbf{u}}

\def\O{\widetilde{\Omega}}

\newcommand{\R}{{\mathbb R}}
\begin{document}

\title{Long-time behavior and weak-strong uniqueness for incompressible
viscoelastic flows}

\author{{\sc Xianpeng Hu}\footnote{Department of Mathematics, City University of Hong Kong, Hong Kong, P.R. China.
E-mail: \textit{xianpehu@cityu.edu.hk}} \
and {\sc Hao Wu}\footnote{School of
Mathematical
Sciences and Shanghai Key Laboratory for Contemporary Applied Mathematics, Fudan University, Shanghai 200433, P.R. China. E-mail: \textit{haowufd@yahoo.com}}}

\date{\today}

\maketitle


\begin{abstract}
We consider the Cauchy problem for incompressible viscoelastic fluids in the whole space $\mathbb{R}^d$ ($d=2,3$).
By introducing a new decomposition via Helmholtz's projections, we first provide an alternative proof on the existence of global smooth solutions near equilibrium. Then under additional assumptions that the initial data belong to $L^1$ and their Fourier modes do not degenerate at low frequencies, we obtain the optimal $L^2$ decay rates for the global smooth solutions and their spatial derivatives. At last, we establish the weak-strong uniqueness property in the class of finite energy weak solutions for the incompressible viscoelastic system.

\end{abstract}

\section{Introduction}
\setcounter{equation}{0}

We study the initial value problem for the incompressible viscoelastic flow
\be\label{1}
\begin{cases}
& \u_t+ \u \cdot \nabla \u-\mu\Delta \u+\nabla p=\nabla \cdot \left(\frac{\partial W(\F)}{\partial \F}\F^T\right),\\
& \nabla \cdot \u=0, \\
& \F_t+ \u\cdot \nabla \F=\nabla \u \F,\\
& \u(t,x)|_{t=0}=\u_0(x), \quad \F(t,x)|_{t=0}=\F_0(x),
\end{cases}
\ee
for $(t, x)\in [0, +\infty)\times \mathbb{R}^d$, $d=2,3$. Here, the vector $\u(t, x): [0, +\infty)\times \mathbb{R}^d\to \mathbb{R}^d$ denotes the velocity field of materials, the scalar function $p(t, x): [0, +\infty)\times \mathbb{R}^d\to \mathbb{R}$ denotes the pressure, the matrix $\F(t, x): [0, +\infty)\times \mathbb{R}^d\to \mathbb{R}^{d\times d}$ stands for the deformation tensor, and $W(\F)$ is the elastic energy functional.
The third equation for $\F$ in \eqref{1} can be regarded as the consistence condition of the flow trajectories obtained from the
velocity field $\u$ and also of those obtained from the deformation tensor $\F$ \cite{LLZ05,LW01}. Moreover, on the right-hand side of the momentum equation, $\frac{\partial W(\F)}{\partial \F}$ is the Piola-Kirchhoff stress tensor and $\frac{\partial W(\F)}{\partial \F}\F^T$ is the Cauchy-Green tensor. The latter is the change variable (from Lagrangian to
Eulerian coordinates) form of the former one \cite{LLZ08}. For the sake of simplicity, we shall confine ourselves to the Hookean linear elasticity such that
$$W(\F)=\frac12|\F |^2,$$
which, however, does not reduce the essential difficulties for analysis. The
results we obtain below can be generalized to more general isotropic elastic energy
functions.

The complicated rheological phenomena of complex fluids is a consequence of interactions between the (microscopic) elastic properties and the (macroscopic) fluid motions \cite{La}. System \eqref{1} serves as an important model for the study on the dynamics of complex fluids. It presents the competition between the elastic energy and the kinetic energy, while the deformation tensor $\F$ carries all the transport/kinematic information of the micro-structures and configurations in complex fluids \cite{LLZ05,LLZ08,Lin12}. There have been many important contributions on the study of classical solutions to the viscoelastic system \eqref{1}. For instance, existence of local classical solutions as well as global classical solutions near-equilibrium of system \eqref{1} in the two-dimensional case has been proved in \cite{LLZ05} (see also \cite{LZ07} for the same result via incompressible limit). Corresponding results in three dimensional case were obtained in \cite{LLZ08} (see also \cite{CZ}). We refer to \cite{LinZ08} for the case of bounded domain, and to \cite{Q10,ZF12} for the well-posedness  in critical functional spaces. Global existence for two-dimensional small-strain viscoelasticity without assumptions on the smallness of the rotational part of the initial deformation tensor was proved in \cite{LLZ07,LZ10}. For more detailed information on the mathematical analysis of the viscoelastic system \eqref{1}, we may refer to the recent review paper \cite{Lin12}. Besides, it is also worth mentioning that there are some recent progress on the compressible viscoelastic flows \cite{DLZ,HuWang10,HuWang11,HuWang12,HWGC,Q,QZ}.

System \eqref{1} is equivalent to the usual Oldroyd-B model for viscoelastic fluids in the case of infinite Weissenberg number \cite{La,LLZ05}, which can be regarded as a coupling of a parabolic system with a hyperbolic one.
It (formally) obeys the following \textit{basic energy law} \cite{LLZ05, LLZ08}:
\be
\frac12\frac{d}{dt}\left(\|\u\|^2_{L^2}+\|\F\|_{L^2}^2\right)+\mu\|\nabla \u\|_{L^2}^2=0,\label{belF}
\ee
which indicates the absence of damping mechanism in the transport equation of $\F$. This brings the main difficulty in the proof of global existence of smooth solutions near equilibrium. Therefore, special treatments will be required to reveal the specific  physical structures of the system. In the two-dimensional case, the authors of \cite{LLZ05} introduced an auxiliary vector field to replace the variable $\F$ and revealed intrinsic nature of weak dissipation for the induced stress tensor. In the three dimensional case, a key observation was made in \cite{LLZ08} that for the strain tensor $\E$ given by $\E=\F-\mathbb{I}$, its curl $\nabla \times \E$ is indeed a higher-order term (see Lemma \ref{curl} below). Then by introducing an auxiliary variable $w=-\Delta \u+\mu^{-1}\nabla \cdot \E$ that extracts certain hidden dissipative nature on $\E$, the authors in \cite{LLZ08} proved global existence of small smooth solutions (see also \cite{CZ} for a different approach).

In this paper, our first aim is to prove the optimal $L^2$ decay rates for global (small) smooth solutions to the incompressible viscoelastic system \eqref{1} (see Theorem \ref{main}). We note that optimal decay rates for weak solutions of incompressible Navier-Stokes equations have been studied extensively in the literature, see, for instance, \cite{Ka,KM,Sch85, Sch86} and the references cited therein. However, due to the lack of dissipation on the deformation tensor $\F$, it seems that the approaches employed in the previously mentioned works (e.g., the Fourier splitting method \cite{Sch85,Sch86}) fail to apply.  On the other hand, it has been observed that the incompressible viscoelastic system \eqref{1} has a structure similar to that of the compressible Navier-Stokes equations (see \cite{Lin12}), although it may be even harder since the equation for the scalar density function (denoted by $\rho$) is now replaced by an equation for a matrix-valued function $\F$. Our method is inspired by the decomposition introduced in \cite{Da00} for the compressible Navier-Stokes equations, with which the global existence of strong solutions near equilibrium in critical spaces was proved. There the velocity field $\u$ is decomposed into the ``compressible part" (denoted by $c$) and the ``incompressible part". Then the associated linearized system for the density $\rho$ and the ``compressible part" of the velocity $c$ is shown to have a parabolic smoothing effect on $\rho$ and $c$ in the low frequencies, and also a damping effect on $\rho$ in the high frequencies. This fact turns out to be crucial in \cite{Da00} to prove the global existence of strong solutions near equilibrium for isentropic compressible fluids. Now for the viscoelastic system \eqref{e1} (i.e., the equivalent form of system \eqref{1} in terms of variables $(\u, \E)$), we introduce a different decomposition on the strain tensor $\E$ (see \eqref{newvar} below), instead on the velocity field like in \cite{Da00}. Then we observe that the linearized system for the velocity $\u$ and the ``compressible part" of $\E$ denoted by $\mathbf{n}=\Lambda^{-1}(\nabla \cdot\E)$  (see \eqref{un}) has a structure similar to the linearized system for $(\rho, c)$ of the compressible Navier-Stokes equations in \cite{Da00}. Besides, the induced coupled system \eqref{OE} for the new variables $\Omega=\Lambda^{-1}(\nabla \u -\nabla^T \u)$ and $\mathbb{E}= \E^T-\E$ provides us some extra dissipation on $\mathbb{E}$ (see Lemma \ref{diin}), which helps us to recover the existence of global smooth solutions near equilibrium (see Proposition \ref{glo}). Based on these special structures for the decomposed systems \eqref{un}--\eqref{OE} that are derived from the original incompressible viscoelastic system \eqref{e1}, we are able to obtain the optimal $L^2$ decay estimates by using similar arguments for the compressible Navier-Stokes equations and related models (see e.g., \cite{Ko02,KS99,LMZ10,LZ11,TW12}, cf. \cite{HWGC} for the decay of the compressible viscoelastic system). Briefly speaking, we first consider the linearized problems (see \eqref{L1} and \eqref{LOE}) of the decomposed systems \eqref{un}--\eqref{OE} near the equilibrium state and investigate the spectrum of the associated semigroups in terms of the decomposition of wave modes at the lower frequency and higher frequency, respectively. Then the decay of the nonlinear system follows from the Duhamel's principle together with the $L^2$ estimates for the corresponding semigroups.

The second result in this paper (see Theorem \ref{unique}) concerns the weak-strong uniqueness of the incompressible viscoelastic system \eqref{1}. Global existence of weak solutions to the incompressible viscoelastic flows with large initial data is an outstanding open problem, although there are some progress in this direction \cite{LM00,LW01}. Recently, the authors \cite{HL} proved the global existence of weak solutions with small energy for \eqref{1} in the two dimensional case $d=2$. However, the existence of weak solutions for the three dimensional case is still open. On the other hand, uniqueness of weak solutions is also a fundamental and important topic in the mathematical theory of incompressible Navier-Stokes equations. The basic idea of the so-called weak-strong uniqueness is the following: assume that a weak solution has some appropriate extra regularity (i.e., being a classical solution in suitable sense), and then prove its uniqueness in the class of weak solutions. The weak-strong uniqueness property for the classical Navier-Stokes equations has been discussed under different regularity assumptions (see e.g., \cite{Giga,KS96,Po,Serrin}). It gives specific conditions under which the Navier-Stokes equations are well-behaved and weak solutions are unique. When dealing with the viscoelastic system \eqref{1}, the main difficulties lie in the partially dissipative structure as well as the nonlinear coupling between the velocity $\u$ and the deformation tensor $\F$.

 The remaining part of this paper is organized as follows.
 In Section 2, we first recall some important features of the incompressible viscoelastic system \eqref{1} and then state our main results (see Theorems \ref{main}, \ref{unique}). In Section 3, we introduce a suitable decomposition of the nonlinear system and provide an alternative proof on the global existence of smooth solutions near equilibrium based on that decomposition.  Section 4 is denoted to the proof of Theorem \ref{main}. We first derive decay estimates for the associated linearized problems of the decomposed systems and then establish the optimal decay rates for the original nonlinear problem. In the last Section 5, we prove Theorem \ref{unique} on the weak-strong uniqueness.


\section{Preliminaries and main results}
\setcounter{equation}{0}
\noindent

We denote by $L^p(\R^d)$, $W^{m,p}(\R^d)$ the usual Lebesgue and Sobolev spaces on $\mathbb{R}^d$, with norms $\|\cdot\|_{L^p}$, $\|\cdot\|_{W^{m,p}}$, respectively. For $p=2$, we denote $H^m(\R^d)= W^{m,2}(\R^d)$ with norm $\|\cdot \|_{H^m}$. The inner product on $L^2$ will be denoted by $(\cdot, \cdot)$. For simplicity, we do not distinguish functional spaces when scalar-valued or vector-valued
functions are involved. We denote by $C$ a generic positive constant throughout
this paper, which may vary at different places. Its special dependence will be indicated explicitly if necessary.

Below we recall some important properties on the structure of the incompressible viscoelastic system \eqref{1}.

First, we note that the incompressibility condition $\nabla \cdot \u=0$ can be exactly represented as
${\rm det}\F =1$. The following well-known result in \cite{Te} illustrates the
incompressible consistence of the viscoelastic system:
\bl\label{inc}  Let $(\u, \F)$ be a
solution of system \eqref{1}. If we assume that ${\rm det}\F_0 =1$, then we have
${\rm det}\F(t,x) =1$, for $t\geq 0$, $x\in \mathbb{R}^d$, $(d=2,3)$.
\el
 Besides, it was shown that the third equation in \eqref{1} has a div-curl structure of compensate compactness \cite{LLZ05,LW01}:
\bl\label{trans}
Let $(\u, \F)$ be a
solution of system \eqref{1}. If we assume that $ \nabla \cdot \F_0^T=0$ then we have
$\nabla \cdot \F^T(t,x)=0$, for $t\geq 0$, $x\in \mathbb{R}^d$, $(d=2,3)$.
\el
When we are interested in small classical solutions, it is convenient to introduce the strain tensor defined by
 \be
 \E=\F-\mathbb{I},
 \label{CH}
 \ee
 where $\mathbb{I}$ is the $d\times d$ identity matrix. Using \eqref{CH}, Lemma \ref{trans} and the following simple fact
\be
\nabla \cdot(\F\F^T)= \nabla \cdot \E + \nabla \cdot \E^T+  \nabla \cdot (\E\E^T)=  \nabla \cdot \E +  \nabla \cdot (\E\E^T),\non
\ee
the original system \eqref{1} can be re-written in terms of $(\u, \E)$:
\be\label{e1}
\begin{cases}
& \u_t+\u \cdot \nabla \u-\mu\Delta \u+\nabla p=\nabla \cdot \E+\nabla \cdot \left(\E \E^T\right), \\
& \nabla \cdot\u=0,\\
& \E_t+ \u\cdot \nabla \E= \nabla \u+ \nabla \u \E, \\
& \u(t,x)|_{t=0}=\u_0(x), \quad \E(t,x)|_{t=0}=\E_0(x)=\F_0(x)-\mathbb{I}.
\end{cases}
\ee%
In \cite{LLZ08}, the authors proved the following lemma, which  indicates that $\nabla \times \E$ is indeed a higher-order term:
\bl\label{curl}
Let $(\u, \E)$ be a solution of system \eqref{e1}. If we assume that
\be
\nabla_m \E_{0ij}-\nabla_j \E_{0im}=\E_{0lj}\nabla_l \E_{0im}-\E_{0lm}\nabla_l \E_{0ij},\non
\ee
then for $t \geq 0$, $x\in \mathbb{R}^d$,  $(d=2,3)$, it holds
\be
\nabla_m \E_{ij}(t,x)-\nabla_j \E_{im}(t,x)=\E_{lj}(t,x)\nabla_l \E_{im}(t,x)-\E_{lm}(t,x)\nabla_l \E_{ij}(t,x).\non
\ee
\el

\begin{remark}
The identity stated in Lemma \ref{curl} is actually the so-called Piola identity, which is equivalent to the commutation of the second derivative of the flow map with respect to the Lagrangian coordinate (see, for instance, \cite{SE,LLZ08}).
\end{remark}

Based on the above observations, in the remaining part of the paper, we will make the following assumptions on the initial data so that Lemmas \ref{inc}--\ref{curl} always apply:
\begin{itemize}
\item[(A1)] $\qquad {\rm det} (\E_0+\mathbb{I})=1, \quad  \nabla\cdot \u_0=0$,
\item[(A2)] $\qquad\nabla\cdot \E_0^T=0$,
\item[(A3)]  $\qquad \nabla_m \E_{0ij}-\nabla_j \E_{0im}=\E_{0lj}\nabla_l \E_{0im}-\E_{0lm}\nabla_l \E_{0ij}$.
\end{itemize}
\medskip

The first result of this paper reads as follows:
\bt[Optimal decay of global classical solutions]\label{main}
Suppose that $d=3$ and the initial
data $\u_0, \E_0\in L^1(\mathbb{R}^3)\cap H^k(\mathbb{R}^3)$ $(k\geq 2$ being an integer$)$ fulfill
the assumptions (A1)--(A3). If the initial data satisfy $\|\u_0\|_{H^2}+\|\E_0\|_{H^2}\leq \delta$ for certain sufficiently small $\delta>0$, then the Cauchy problem \eqref{e1} admits a unique global classical solution $(\u, \E)$ such that
\be
\begin{cases}
 \partial_t^j\nabla^\alpha \u \in L^\infty(0,T; H^{k-2j-|\alpha|}(\mathbb{R}^3))\cap L^2(0,T; H^{k-2j-|\alpha|+1}(\mathbb{R}^3)),\\
 \partial_t^j\nabla^\alpha \E \in L^\infty(0,T; H^{k-2j-|\alpha|}(\mathbb{R}^3)),
 \end{cases}
 \non
\ee
for all integer $j$ and multi-index $\alpha$ satisfying $2j+|\alpha|\leq k$.

For all $t \geq 0$, the following decay estimates hold
\bea
 \|\u (t)\|_{L^2}+\| \E(t) \|_{L^2}&\leq & C M(1+t)^{-\frac34},\label{deL2a}\\
 \|\nabla \u (t)\|_{H^1}+\|\nabla  \E(t) \|_{H^1} &\leq& C M(1+t)^{-\frac54},\label{deH2}
\eea
where $M=\|\u_0\|_{L^1\cap H^2}+ \|\E_0\|_{L^1\cap H^2}$.

Moreover, if the Fourier transforms of the initial data $(\u_0, \mathbf{n}_0)$ $($where $\mathbf{n}_0=\Lambda^{-1} \nabla \cdot \E_0$, and the operator $\Lambda$ is defined in \eqref{LLL} below$)$ also satisfy $|\widehat{u_{i0}}|\geq c_0$, $|\widehat{n_{i0}}|\geq c_0$ for $0\leq |\xi|<<1$, where the lower bound $c_0>0$ satisfies $c_0\sim O(\delta^\zeta)$ with $\zeta\in (0,1)$, then there exists a $t_0>>1$ such that
\be
\|\u (t)\|_{L^2}+\| \E(t) \|_{L^2}\geq C(1+t)^{-\frac34},\quad \forall\, t\geq t_0,\label{deL2b}
\ee
i.e., the $L^2$ decay rate \eqref{deL2a} is optimal.
\et
\begin{remark}
For the sake of simplicity, we only give the proof for the three dimensional case $(d=3)$ in the subsequent text. Under the same assumptions on the initial data as in Theorem \ref{main}, one can easily obtain the corresponding optimal decay estimates in the two dimensional case $(d=2)$, namely,
\bea
 \|\u (t)\|_{L^2}+\| \E(t) \|_{L^2} &\leq& CM(1+t)^{-\frac12},\non\\
 \|\nabla \u (t)\|_{H^1}+\|\nabla  \E(t) \|_{H^1} &\leq& CM(1+t)^{-1}.\non
\eea
\end{remark}

Next, we introduce the notion of weak solutions to the system \eqref{e1}. Denote $\F_j$ ($j=1,...,d$) be the columns of $\F$. Then the third equation in \eqref{1} for $\F$ can be written as
\be
\partial_t\F_j+\u\cdot\nabla\F_j=\F_j\cdot\nabla\u,\quad j=1,...,d.\label{F}
\ee
  Recalling Lemma \ref{trans}, we have $\nabla \cdot \F^T=0$, which yields that $\nabla \cdot \F_j=0$. Then  \eqref{F} can be further re-written in terms of $\E=\F-\mathbb{I}$ such that
 \be
 \partial_t\E_j+\nabla \cdot(\E_j\otimes\u-\u\otimes\E_j)=\nabla_j \u,\label{Ej}
  \ee
where $(a\otimes b)_{ij}=a_ib_j$.

On the other hand, we note that the basic energy law \eqref{belF} can be re-written in terms of $(\u, \E)$:
\be
\frac12\frac{d}{dt}\left(\|\u\|^2_{L^2}+\|\E\|_{L^2}^2\right)+\mu\|\nabla \u\|_{L^2}^2=0.\label{bel}
\ee
\begin{definition}\label{defw}
We say that the pair $(\u,\E)$ is a finite energy weak solution to the system \eqref{e1} in $(0,T)\times \R^d$, if the following hold:
 \bea
 && \u \in L^\infty(0,T; L^2(\R^d)), \quad \nabla \u\in L^2(0,T; L^2(\R^d)),\nonumber\\
 && \E=\F-\mathbb{I}\in L^\infty(0,T; L^2(\R^d)).\nonumber
 \eea
 The pair $(\u, \E)$ satisfies \eqref{e1} in the sense of distributions, i.e., $\nabla \cdot \u=0$  in the distributional sense and for all test functions $\mathbf{v}, \mathbf{w} \in  C^1([0,T]; \mathcal{D}(\R^d))$ with $\nabla \cdot \mathbf{v}=0$, we have
 \be
 \begin{split}
&\int_{\R^d}\u(t,\cdot)\cdot\mathbf{v}(t,\cdot)dx -\int_{\R^d}\u_0\cdot\mathbf{v}(0,\cdot)dx\\
&\ = \int_0^t\int_{\R^d}\left[\u\cdot\mathbf{v}_t+\left(\u\otimes\u-\E\E^T-\E\right):\nabla\mathbf{v}\right] dxd\tau\\
&\quad \ \ -\mu\int_0^t\int_{\R^d}\nabla\u:\nabla\mathbf{v} dxd\tau,\label{w1}
\end{split}
\ee
\be
\begin{split}
& \int_{\R^d}\E_j(t,\cdot)\cdot\mathbf{w}(t,\cdot)dx-\int_{\R^d}\E_j(0,\cdot)\cdot \mathbf{w}(0,\cdot)dx\\
&\ = \int_0^t\int_{\R^d}\big[\E_j\mathbf{w}_t+\left(\E_j\otimes\u-\u\otimes\E_j\right):\nabla\mathbf{w} -\u\cdot \nabla _j \mathbf{w}\big]dxd\tau,\label{w2}
\end{split}
\ee
for $j=1,...,d$ and $t\in [0,T]$. Moreover, the following energy inequality holds for a.e. $t\in [0,T]$, that is
\begin{equation}\label{w3}
\mathcal{E}(t)+\mu\int_0^t\|\nabla \u(s)\|_{L^2}^2ds
\le \mathcal{E}(0),
\end{equation}
where
\be
\mathcal{E}(t)=\frac12\left(\|\u(t)\|^2_{L^2}+\|\E(t)\|_{L^2}^2\right), \quad \mathcal{E}(0)=\frac12\left(\|\u_0\|^2_{L^2}+\|\E_0\|_{L^2}^2\right).\non
\ee
\end{definition}
Now we state our second result of this paper.

\bt[Weak-strong uniqueness]\label{unique}
Suppose $d=2,3$. Let $\u_0\in H^k(\mathbb{R}^d)$ and $\E_0=\F_0-\mathbb{I}\in H^k(\mathbb{R}^d)$ $(k\geq 3)$, satisfying the assumptions (A1)--(A3). Suppose that $(\hat{\u},\hat{\E})$ is a weak solution of system \eqref{e1} (in the sense of Definition \ref{defw}) in $(0,T)\times \R^d$ and $(\u,\E)$ is a strong solution emanating from the same initial data (e.g., the global solutions as in Proposition \ref{glo} below, or the local solutions as in \cite[Theorem 1]{LLZ08}). Then we have
$\hat{\u}\equiv\u$ and $\hat{\E}\equiv\E$ on the time interval of existence.
\et
%


\section{Well-posedness via a new decomposition method}
\setcounter{equation}{0}
\subsection{Decomposed system}\label{decsys}
\noindent
 We write the system \eqref{e1} into the following form
\be\label{e2}
\begin{cases}
& \u_t-\mu\Delta \u-\nabla \cdot \E=\mathbf{g}, \\
& \nabla \cdot\u=0,\\
& \E_t-\nabla \u = \mathbf{h}, \\
& \u(t,x)|_{t=0}=\u_0(x), \quad \E(t,x)|_{t=0}=\E_0(x)=\F_0(x)-\mathbb{I}.
\end{cases}
\ee
The nonlinearities $\mathbf{g}$ and $\mathbf{h}$ are given by
\be
\mathbf{g}=-\mathbb{P} (\u \cdot \nabla \u) + \mathbb{P}\left(\nabla \cdot \left(\E \E^T\right)\right), \quad \mathbf{h}= -\u\cdot \nabla \E+\nabla \u \E,\label{gh}
\ee
respectively, where
$$\mathbb{P}=\mathbb{I}-\Delta^{-1}(\nabla \otimes\nabla)$$
is the Leray projection operator.
\begin{remark}
Due to Lemma \ref{trans}, we note that $\nabla\cdot(\nabla \cdot\E)=\nabla \cdot(\nabla \cdot \E^T)=0$, which implies   $\mathbb{P}(\nabla\cdot \E)=\nabla \cdot \E$.
\end{remark}

  Let $\Lambda^s$ be the pseudo differential operator defined by
 \be
 \Lambda^s f=\mathcal{F}^{-1}(|\xi|^s\widehat{f}(\xi)),\quad s\in \mathbb{R}.\label{LLL}
 \ee
  In particular, it holds $\Lambda^2=-\Delta$.

  Then we introduce the following new variables
 \be
 \begin{cases}
 &\mathbf{n}=\Lambda^{-1}(\nabla \cdot\E),\\
 & \Omega=\Lambda^{-1}(\nabla \u -\nabla^T \u),\\
 &  \mathbb{E}= \E^T-\E.
 \end{cases}
 \label{newvar}
 \ee
It follows from \eqref{e2} that the pairs of new variables $(\mathbf{u}, \mathbf{n})$ and $(\Omega, \mathbb{E})$ satisfy the following nonlinear systems (i.e., the ``decomposed systems" mentioned in the introduction), respectively:
\be\label{un}
\begin{cases}
&\u_t-\mu\Delta \u-\Lambda \mathbf{n}=\mathbf{g},\\
&\mathbf{n}_t+\Lambda\u=\Lambda^{-1}\nabla \cdot \mathbf{h},
\end{cases}
\ee
and
\be\label{OE}
\begin{cases}
& \Omega_t-\mu\Delta \Omega-\Lambda \mathbb{E}=\Lambda^{-1}(\nabla \mathbf{g}-\nabla^T\mathbf{g})+\Lambda^{-1}\S,\\
& \mathbb{E}_t+\Lambda\Omega=  \mathbf{h}^T-\mathbf{h},
\end{cases}
\ee
In \eqref{OE}, we have used the following fact from Lemma \ref{curl} such that
\be
\begin{split}
(\nabla \nabla\cdot\E-\nabla^T \nabla\cdot\E)_{ij}
&= \nabla_j\nabla_k\E_{ik}-\nabla_i\nabla_k\E_{jk}\\
&= \nabla_k\nabla_j\E_{ik}-\nabla_k\nabla_i\E_{jk}\\
&= \nabla_k(\nabla_k \E_{ij}+\E_{lk}\nabla_l\E_{ij}-\E_{lj}\nabla_l\E_{ik})\\
&\quad\ -\nabla_k(\nabla_k \E_{ji}+ \E_{lk}\nabla_l\E_{ji}-\E_{li}\nabla_l\E_{jk})\\
&= \Delta (\E_{ij}-\E_{ji}) +\S_{ij},\non
\end{split}
\ee
where the higher-order antisymmetric matrix $\S$ is given by
\be
\S_{ij}=\nabla_k(\E_{lk}\nabla_l\E_{ij}-\E_{lj}\nabla_l E_{ik})-\nabla_k(E_{lk}\nabla_l\E_{ji}-\E_{li}\nabla_l\E_{jk}).\non
\ee

%
The following elementary estimates will be useful in the subsequent proofs.
\bl\label{pre}
Let the assumptions (A1)--(A3) be satisfied. The solution $\E$ to system \eqref{e1} satisfy the following estimates:
\bea
 \|\E\|_{L^2} &\leq& C\left(\|\mathbf{n}\|_{L^2}+\|\E\|_{H^2}\|\E\|_{L^2}\right),\label{EL2}\\
  \|\nabla \E\|_{L^2} &\leq&  C\left(\|\nabla \mathbf{n}\|_{L^2}+\|\E\|_{H^2}\|\nabla \E\|_{L^2}\right),\label{EnL2}\\
 \|\Delta \E\|_{L^2} & \leq&  C\left(\|\Delta \mathbb{E}\|_{L^2}+C\|\E\|_{H^2}\|\Delta\E\|_{L^2}\right),\label{EEE}
\eea
where $C$ is a constant that does not depend on $\E$.
\el
\begin{proof}
We deduce from the Hodge decomposition
\be \Delta \E=\nabla \nabla \cdot \E-\nabla \times\nabla \times \E,
\label{Hodge}
\ee
Lemma \ref{curl} and the Sobolev embedding theorem that
\be
\begin{split}
\|\E\|_{L^2}&=\|\Lambda^{-2}\Delta \E\|_{L^2}\\
&=\|\Lambda^{-2}(\nabla \nabla \cdot\E-\nabla \times\nabla \times \E)\|_{L^2}\\
&= \|\Lambda^{-1}\nabla \mathbf{n}-\Lambda^{-2} \nabla \times(\nabla \times \E)\|_{L^2}\\
&\leq C(\|\mathbf{n}\|_{L^2}+\|\nabla \E\|_{L^3}\|\E\|_{L^2})\\
&\leq C(\|\mathbf{n}\|_{L^2}+\|\E\|_{H^2} \|\E\|_{L^2}),\non
\end{split}
\ee
and
\be
\begin{split}
\|\nabla \E\|_{L^2}&\leq C\|\nabla \mathbf{n}-\Lambda^{-1}\nabla \times (\nabla \times \E)\|_{L^2}\\
&\leq  C\left(\|\nabla \mathbf{n}\|_{L^2}+\|\E\|_{L^\infty}\|\nabla \E\|_{L^2}\right)\\
&\leq C\left(\|\nabla \mathbf{n}\|_{L^2}+\|\E\|_{H^2} \|\nabla \E\|_{L^2}\right),\non
\end{split}
\ee
which yield the conclusions \eqref{EL2} and \eqref{EnL2}.

Next, from Lemmas \ref{trans}, \ref{curl}, we deduce that
\be
\begin{split}
\|\Delta \E\|_{L^2}&\leq \|\nabla \nabla \cdot \E\|_{L^2}+\|\nabla \times\nabla \times \E\|_{L^2}\\
&= \|\nabla \nabla \cdot \mathbb{E}\|_{L^2}+ \|\nabla \times\nabla \times \E\|_{L^2}\\
&\leq \left(\|\Delta \mathbb{E}\|_{L^2}+ \|\nabla \times\nabla \times \mathbb{E}\|_{L^2}\right) + \|\nabla \times\nabla \times \E\|_{L^2}\\
&\leq  C\|\Delta \mathbb{E}\|_{L^2}+C (\|\E\|_{L^\infty}\|\Delta \E\|_{L^2}+\|\nabla \E\|_{L^3}\|\nabla \E\|_{L^6})\\
&\leq C\left(\|\Delta \mathbb{E}\|_{L^2}+\|\E\|_{H^2}\|\Delta\E\|_{L^2}\right),\non
\end{split}
\ee
which gives \eqref{EEE}. The proof is complete.
\end{proof}

\subsection{Global strong solution revisited}

Global existence of smooth solutions near equilibrium to system \eqref{e1} has been proved in \cite[Theorem 2]{LLZ08}. In what follows, we provide an alternative proof based on the decomposed systems \eqref{un} and \eqref{OE} introduced in Section \ref{decsys}, which may have its own interest.

First, we present some simple estimates on the nonlinearities $\mathbf{g}$ and $\mathbf{h}$ in \eqref{e2}.

\bl\label{pre1}
Suppose that the assumptions (A1)--(A3) are satisfied. We have the following estimates on nonlinear terms $\mathbf{g}$ and $\mathbf{h}$:
\be
 \|\mathbf{g}\|_{L^2} \leq  C(\|\u\|_{H^1}+\|\E\|_{H^1})(\|\Delta \u\|_{L^2}+\|\Delta \E\|_{L^2}),\label{egL12}
 \ee
 \be
  \|\nabla \mathbf{g}\|_{L^2} \leq  C(\|\u\|_{H^2}+\|\E\|_{H^2})(\|\nabla^2\u\|_{L^2}+\|\nabla^2\E\|_{L^2}),\label{ggg}
 \ee
 \be
  \|\Lambda^{-1}\nabla \cdot \mathbf{h}\|_{L^2}+\|\mathbf{h}\|_{L^2} \leq C  (\|\u\|_{H^1}+\|\E\|_{H^1})(\|\Delta \u\|_{L^2}+\|\Delta \E\|_{L^2}),  \label{ehL12}
  \ee
  \be
 \|\nabla \Lambda^{-1}\nabla \cdot \mathbf{h}\|_{L^2} \leq C(\|\u\|_{H^2}+\|\E\|_{H^2})(\|\Delta \u\|_{L^2}+\|\Delta \E\|_{L^2}),\label{ehH1}
 \ee
where the constant $C$ does not depend on $(\u, \E)$.
\el
\begin{proof}
Using the Sobolev embedding theorem and the fact that the Leray-Hopf projection $\mathbb{P}$ is a bounded operator from  $L^2$ to $L^2$, we infer from the definitions of $\mathbf{g}$ and $\mathbf{h}$ that
\be
\begin{split}
\|\mathbf{g}\|_{L^2}&\leq C\|\u\|_{L^3}\|\nabla \u\|_{L^6}+C\|\E\|_{L^3}\|\nabla \E\|_{L^6}\\
&\leq C(\|\u\|_{H^1}+\|\E\|_{H^1})(\|\Delta \u\|_{L^2}+\|\Delta \E\|_{L^2}),\non
\end{split}
\ee
\be
\begin{split}
 \|\nabla \mathbf{g}\|_{L^2}
&\leq C(\|\u\|_{L^\infty}\|\nabla^2 \u\|_{L^2}+\|\nabla \u\|_{L^3}\|\nabla \u\|_{L^6})\\
&\quad \ + C(\|\E\|_{L^\infty}\|\nabla^2 \E\|_{L^2}+\|\nabla \E\|_{L^3}\|\nabla \E\|_{L^6})\\
&\leq C(\|\u\|_{H^2}+\|\E\|_{H^2})(\|\nabla^2\u\|_{L^2}+\|\nabla^2\E\|_{L^2}),\non
\end{split}
\ee
\be
\begin{split}
 \|\Lambda^{-1}\nabla \cdot \mathbf{h}\|_{L^2}+\|\mathbf{h}\|_{L^2}
 &\leq C \|\u\|_{L^3}\|\nabla \E\|_{L^6}+C\|\nabla \u\|_{L^6}\|\E\|_{L^3}\\
 &\leq C(\|\u\|_{H^1}+\|\E\|_{H^1})(\|\Delta \u\|_{L^2}+\|\Delta \E\|_{L^2}),\non
 \end{split}
\ee
\be
\begin{split}
\|\nabla \Lambda^{-1}\nabla \cdot \mathbf{h}\|_{L^2}
&\leq C(2\|\nabla \u\|_{L^3}\|\nabla \E\|_{L^6}+\|\u\|_{L^\infty}\|\nabla^2 \E\|_{L^2}+\|\E\|_{L^\infty}\|\nabla^2\u\|_{L^2})\\
&\leq C(\|\u\|_{H^2}+\|\E\|_{H^2})(\|\Delta \u\|_{L^2}+\|\Delta \E\|_{L^2}).\non
\end{split}
\ee
Hence, the proof is complete.
\end{proof}

Next, we derive some higher-order differential inequalities in terms of $(\u, \E)$ as well as the new variables $(\Omega, \mathbb{E})$:

\bl\label{diin}
Let $(\u, \E)$ be a smooth solution to system \eqref{e1}. Then the following inequalities hold:
\be
\begin{split}
& \frac12\frac{d}{dt}\left(\|\Delta \u\|_{L^2}^2+\|\Delta \E\|_{L^2}^2\right)
   +\mu\|\nabla \Delta \u\|_{L^2}^2\\
&\ \leq C\left(\|\u\|_{H^2}+\|\E\|_{H^2}\right)\left(\|\Delta \E\|^2_{L^2}+\|\nabla \u\|^2_{L^2}+\|\nabla \Delta \u\|_{L^2}^2\right),\label{h3}
\end{split}
\ee
and
\be
\begin{split}
 \frac{d}{dt}(\Lambda\Omega, \Delta \mathbb{E})+\frac12 \|\Delta \mathbb{E}\|_{L^2}^2
&\leq C\left(\|\nabla \u\|_{L^2}^2+\|\nabla \Delta \u\|_{L^2}^2\right)\\
&\quad\ + C\left(\|\u\|_{H^2}^2+\|\E\|_{H^2}^2\right)\left(\|\Delta\u\|^2_{L^2}+\|\Delta \E\|^2_{L^2}\right).
\label{h4}
\end{split}
\ee
\el
\begin{proof}
We note that the first differential inequality  \eqref{h3} has been obtained in \cite{LLZ08} (see (73) therein). Next, we denote $$\widetilde{\Omega}= \Lambda\Omega= \nabla \u-\nabla^T \u.$$ Then for the system \eqref{OE} in terms of $(\Omega, \mathbb{E})$, by a direct computation, we get
\be
\begin{split}
 \frac{d}{dt}(\widetilde{\Omega}, \Delta \mathbb{E}) &=  (\widetilde{\Omega}_t, \Delta \mathbb{E})+ (\widetilde{\Omega}, \Delta \mathbb{E}_t)\\
&= \mu(\Delta \widetilde{\Omega}, \Delta \mathbb{E})-\|\Delta \mathbb{E}\|_{L^2}^2+(\nabla \mathbf{g}-\nabla^T \mathbf{g}, \Delta \mathbb{E})+(\S, \Delta \mathbb{E})\\
&\quad -(\O, \Delta \O)+(\Delta \O, \mathbf{h}^T-\mathbf{h}).
\label{cross}
\end{split}
\ee
The Sobolev embedding theorem yields the following estimate on $\S$:
\be
\begin{split}
\|\S\|_{L^2} &\leq  \|\E\|_{L^\infty}\|\nabla^2 \E\|_{L^2}+\|\nabla \E\|_{L^3}\|\nabla \E\|_{L^6}\\
&\leq  C\|\E\|_{H^2}\|\nabla^2\E\|_{L^2}.\non
\end{split}
\ee
Using the above estimate, \eqref{ggg}, \eqref{ehL12} (see Lemma \ref{pre1}),  the H\"older inequality and Young's inequality, we infer from \eqref{cross} that
\be
\begin{split}
& \frac{d}{dt}(\widetilde{\Omega}, \Delta \mathbb{E})+\frac12\|\Delta \mathbb{E}\|_{L^2}^2\\
&\ \leq C\|\nabla \u\|_{H^2}^2+ C\left(\|\u\|_{H^2}^2+\|\E\|_{H^2}^2)(\|\nabla^2\u\|^2_{L^2}+\|\nabla^2\E\|^2_{L^2}\right),\non
\end{split}
\ee
which gives \eqref{h4}. The proof is compete.
\end{proof}

Now we prove the global existence result for initial data that are near equilibrium.
\begin{proposition}\label{glo}
Suppose that the initial
data $\u_0, \E_0\in H^k(\mathbb{R}^3)$ $(k\geq 2$ being an integer$)$ satisfy
the assumptions (A1)--(A3). There exists a sufficiently small constant $\delta_0\in (0,1)$ such that if $ \|\u_0\|_{H^2}+\|\E_0\|_{H^2}\leq \delta\in(0,\delta_0]$, then the system \eqref{e1} admits a unique global
strong solution $(\u, \E)$ such that
\be
\begin{cases}
 \partial_t^j\nabla^\alpha \u \in L^\infty(0,T; H^{k-2j-|\alpha|}(\mathbb{R}^3))\cap L^2(0,T; H^{k-2j-|\alpha|+1}(\mathbb{R}^3)),\\
 \partial_t^j\nabla^\alpha \E \in L^\infty(0,T; H^{k-2j-|\alpha|}(\mathbb{R}^3)),
 \end{cases}\label{regglo}
\ee
for all integers $j$ and multi-index $\alpha$ satisfying $2j+|\alpha|\leq k$.
Moreover, we have
\be
\|\u(t)\|_{H^2}+\|\E(t)\|_{H^2}\leq 4\delta, \quad \forall\, t\geq 0\label{stability}
\ee
and
\be
\int_0^{+\infty} \|\nabla \Delta \u(t)\|^2_{L^2} dt\leq C, \label{regcri}
\ee
where the constant $C$ depends on $\delta$.
\end{proposition}
\begin{proof}
Denote
\be
G(t)=\frac12\left(\|\u\|_{L^2}^2+\|\E\|_{L^2}^2+\|\Delta\u\|_{L^2}^2+\|\Delta \E\|_{L^2}^2\right)+\kappa(\Lambda\Omega, \Delta \mathbb{E}),\label{Ge}
\ee
where $\kappa>0$ is a small constant to be determined later. We infer from \eqref{bel}, \eqref{h3} and \eqref{h4} that
\be
\begin{split}
& \frac{d}{dt} G(t)+(\mu-C_1\kappa) \left(\|\nabla \u\|_{L^2}^2+\|\nabla \Delta \u\|_{L^2}^2\right)+ \frac{\kappa}{2}\|\Delta \mathbb{E}\|_{L^2}^2\\
&\ \leq C_2\left(\|\u\|_{H^2}+\|\E\|_{H^2}+\|\u\|_{H^2}^2+\|\E\|_{H^2}^2\right)\\
&\quad \ \ \times\left(\|\Delta \E\|^2_{L^2}+\|\nabla \u\|^2_{L^2}+\|\nabla \Delta \u\|_{L^2}^2\right),\label{hi}
\end{split}
\ee
for some constants $C_1, C_2$ independent of the variables $\u$, $\E$, $\Omega$ and $\mathbb{E}$.

It easily follows from \eqref{newvar} that $\|\Delta \mathbb{E}\|_{L^2}\leq 2\|\Delta \E\|_{L^2}$ and $\|\Lambda\Omega\|_{L^2}\leq 2 \|\nabla \u\|_{L^2}$. Thus, we can take $\kappa>0$ sufficiently small in \eqref{Ge} such that
\be
\|\u\|_{H^2}^2+\|\E\|_{H^2}^2\geq G(t)\geq \frac14(\|\u\|_{H^2}^2+\|\E\|_{H^2}^2) \quad \text{and}\quad C_1\kappa\leq \frac{\mu}{2}.\label{Ge1}
\ee
On the other hand, we have shown the relation \eqref{EEE}, which combined with \eqref{hi} and \eqref{Ge1} yields that
\be
\begin{split}
& \frac{d}{dt} G(t)+ \Big[C_3 -C_4\left(G(t)+G^\frac12(t)\right)\Big]\left(\|\Delta \E\|^2_{L^2}+\|\nabla \u\|_{L^2}^2+\|\nabla \Delta \u\|_{L^2}^2\right)\\
&\ \ \leq 0,\label{hia}
\end{split}
\ee
where the constants $C_3, C_4$ may depend on $C_1, C_2$ and $\mu$.

By the classical continuation argument, we see that if $\delta_0\in (0, 1)$ is sufficiently small, e.g., it satisfies
$C_4(\delta_0^2+\delta_0)<C_3$, then for any $\delta\in (0, \delta_0]$, we can find certain $T_0>0$ depending on $\delta$ such that $$\sup_{t\in[0,T_0]}C_4\left(G(t)+G^\frac12(t)\right)< C_3$$ and it follows from the differential inequality \eqref{hia} that $G(t)\leq G(0)\leq \delta^2$ for $t\in[0,T_0]$.

Let $T^*=\sup T_0$ such that $C_4(G(T^*)+G^\frac12(T^*))= C_3$. If $T^*<+\infty$, then we have $G(T^*)\leq G(0)\leq \delta^2$ and $$C_4(G(T^*)+G^\frac12(T^*))\leq  C_4(\delta_0^2+\delta_0)<C_3,$$
 which yields  a contradiction with the definition of $T^*$. Thus, $T^*=+\infty$ and we can conclude that $G(t)\leq \delta^2 $ for all $t\geq 0$. This together with the relation \eqref{Ge1} yields the uniform estimate \eqref{stability}. Then integrating \eqref{hia} with respective to time, we conclude \eqref{regcri}.

 It follows from \eqref{stability} and \eqref{regcri} that for any $T\in (0,+\infty)$,
\be
\int_0^T \|\nabla \u(t)\|_{H^2}^2 dt\leq C_T,\label{nonbl}
\ee
where $C_T$ is a constant depending on $\|\u_0\|_{H^2}$, $\|\E_0\|_{H^2}$ and $T$. Then by the local existence result and the blow-up criterion obtained in \cite[Theorem 1]{LLZ08}, together with the estimate \eqref{nonbl}, we conclude that $(\u,\E)$ is a global classical solution to system \eqref{e1} satisfying the regularity \eqref{regglo}.  The proof is complete.
\end{proof}

\begin{remark}
Proposition \ref{glo} also indicates that the incompressible viscoelastic  system \eqref{e1} is locally Lyapunov stable (see \eqref{stability}).
\end{remark}

\section{Optimal $L^2$ decay rates}
\setcounter{equation}{0}

\subsection{Decay estimates for linearized systems}
\noindent We first study the decay properties of the associated linearized system of the decomposed system \eqref{un} given by
\be\label{L1}
\begin{cases}
& \u_{t}-\mu \Delta \u-\Lambda \mathbf{n}=0, \\
& \mathbf{n}_{t}+\Lambda \u =0.
\end{cases}
\ee
Denote
\be
\mathbf{W}_i=(u_i, n_i), \quad i=1,2,3, \non
\ee
where $u_i$ and $n_i$ is the $i$-th component of the vectors $\u$ and $\mathbf{n}$, respectively. Then the linear problem  \eqref{L1} can be written into the following form
\be
\begin{cases}
&\mathbf{W}_{it}= \mathcal{B}\mathbf{W}_i,\\
& \mathbf{W}_i(t,x)|_{t=0}=(   u_{i0}, n_{i0})^T,
\end{cases}
 \label{S1a}
\ee
for $i=1,2,3$, where $\mathcal{B}$ is a matrix-valued differential operator given by
\be
\mathcal{B}= \left( \begin{array}{cc}
 \mu\Delta & \Lambda\\
 -\Lambda & 0
 \end{array} \right).\label{B}
\ee
Thus, instead of studying the eigenvalues of the six-dimensional problem \eqref{L1} for $(\u, \mathbf{n})$, we only have to consider a two-dimensional one \eqref{S1a} with the simple operator $\mathcal{B}$.

Applying the Fourier transform to system \eqref{S1a}, we have for $i=1,2,3$,
\be
\widehat{\mathbf{W}}_{it}= \mathcal{A}(\xi)\widehat{\mathbf{W}}_i,\label{S2a}
\ee
where
$\widehat{\mathbf{W}}(t,\xi)=\mathcal{F}[\mathbf{W}(t,x)]$, $\xi=(\xi_1, \xi_2,\xi_3)^T$ and $\mathcal{A}(\xi)$ is the Fourier transformation of the differential operator $\mathcal{B}$ given by
\be
\mathcal{A}(\xi)= \left( \begin{array}{cc}
 -\mu|\xi|^2& |\xi| \\
  -|\xi| & 0
 \end{array} \right).\non
\ee
We can compute the eigenvalues of $\mathcal{A}$ by solving the characteristic equation
\be
{\rm det} (\lambda \mathbb{I}-\mathcal{A}(\xi))= \lambda^2 +\mu|\xi|^2 \lambda+|\xi|^2=0,\non
\ee
such that
\be\label{lambda1}
\lambda_{\pm}(\xi)=
\begin{cases}
-\frac{\mu}{2}|\xi|^2\pm \frac{\i}{2}\sqrt{-\mu^2|\xi|^4+4|\xi|^2},\quad\, \text{if}\ |\xi|<\frac{2}{\mu},\\
-\frac{\mu}{2}|\xi|^2\pm \frac{1}{2}\sqrt{\mu^2|\xi|^4-4|\xi|^2},\qquad\text{if}\ |\xi|\geq \frac{2}{\mu}.
\end{cases}
\ee
Then the associated linear semigroup $e^{t\mathcal{A}}$ can be expressed as
\be
e^{t\mathcal{A}(\xi)}=e^{\lambda_+ t} P_+ + e^{\lambda_- t} P_-,\non
\ee
where the projection operators $P_\pm$ are given by
 \be
 P_+=\frac{\mathcal{A}(\xi)-\lambda_- I}{\lambda_+-\lambda_-}, \quad P_-=\frac{\mathcal{A}(\xi)-\lambda_+ I}{\lambda_- -\lambda_+}.\non
 \ee

 As a result, we have the exact form of the Fourier transform
$\widehat{G}(t,\xi)$ for Green's function $G(t,x)=e^{t \mathcal{B}}$ such that
 \be
 \begin{split}
 \widehat{G}(t,\xi)&:= e^{t\mathcal{A}(\xi)}\\
 &\,= \left( \begin{array}{cc}
  -\mu|\xi|^2\frac{ (e^{\lambda_+ t}- e^{\lambda_- t})}{\lambda_+-\lambda_-}
  - \frac{ \lambda_- e^{\lambda_+ t}- \lambda_+e^{\lambda_- t}}{\lambda_+-\lambda_-}
  &   |\xi| \frac{ (e^{\lambda_+ t}- e^{\lambda_- t})}{\lambda_+-\lambda_-} \\
  -|\xi|\frac{ (e^{\lambda_+ t}- e^{\lambda_- t})}{\lambda_+-\lambda_-}
  & - \frac{ \lambda_- e^{\lambda_+ t}- \lambda_+e^{\lambda_- t}}{\lambda_+-\lambda_-}
 \end{array} \right).
 \label{G1}
 \end{split}
 \ee
 From the expression of $\lambda_\pm$ (see \eqref{lambda1}) and using the simple asymptotic facts
 \be
 \begin{split}
 & \lim_{|\xi|\to+\infty} -\frac{\mu}{2} |\xi|^2+\frac{\sqrt{\mu^2  |\xi|^4-4 |\xi|^2}}{2}\\
 &\quad =\lim_{|\xi|\to+\infty}\frac{2 |\xi|^2}{-\mu  |\xi|^2-\sqrt{\mu^2  |\xi|^4-4 |\xi|^2}}=-\frac{1}{\mu},\non
 \end{split}
 \ee
 \be
  \lim_{|\xi|\to+\infty} -\frac{\mu}{2} |\xi|^2-\frac{\sqrt{\mu^2  |\xi|^4-4 |\xi|^2}}{2}=-\infty,\non
 \ee
 we can verify the approximation of Green's function $G(t,x)$ for both low frequencies and high frequencies:
 \be
 \frac{ e^{\lambda_+ t}-e^{\lambda_- t}}{\lambda_+-\lambda_-}\sim
 \left\{ \begin{array}{cc} e^{-\frac{\mu}{2}|\xi|^2t} \frac{\sin bt}{b}& |\xi|<\eta,\\
 O(1) |\xi|^{-2}e ^{-\gamma t} & |\xi|\geq \eta,
 \end{array} \right.\label{la1}
 \ee
 \be
 \frac{ \lambda_- e^{\lambda_+ t}- \lambda_+e^{\lambda_- t}}{\lambda_+-\lambda_-}\sim
  \left\{ \begin{array}{cc}
 e^{-\frac{\mu}{2}|\xi|^2t} \left(-\cos b t+\frac{\mu}{2b}|\xi|^2 \sin bt\right) & |\xi|<\eta, \\
  O(1) e ^{-\gamma t} & |\xi|\geq \eta,
 \end{array} \right.\label{la2}
 \ee
 where $\eta\in (0, \frac{2}{\mu})$ is a constant that can be chosen sufficiently small,  the constant $\gamma>0$ may depend on $\mu$ and $\eta$, while $b>0$ is a real number such that
 \be
 b=\frac{1}{2}\sqrt{-\mu^2|\xi|^4+4|\xi|^2}\sim |\xi|+O(|\xi|^3), \quad \text{for} \ |\xi|<<1.\label{b}
 \ee
\bl\label{linear1}
 Let $(\u(t), \mathbf{n}(t))$ be a solution to the linear system \eqref{L1} with initial data $(\u_0, \mathbf{n}_0)\in H^l(\mathbb{R}^3)\cap L^1 (\mathbb{R}^3)$.

 (i) For $0\leq |\alpha|\leq l$ and $t\geq 0$, we have
 \be
 \begin{split}
  & \|(\partial_x^\alpha \u(t), \partial_x^\alpha \mathbf{n}(t))\|_{L^2}\\
  &\ \leq C(1+t)^{-\frac34-\frac{|\alpha|}{2}}\left(\|(\u_0, \mathbf{n}_0)\|_{L^1}+\|\partial_x^\alpha (\u_0, \mathbf{n}_0)\|_{L^2}\right).\label{decun}
  \end{split}
 \ee

  (ii) Assume that $\widehat{\u_0}$ and $\widehat{\mathbf{n}_0}$ satisfy $|\widehat{u_{i0}}|\geq c_0>0$, $|\widehat{n_{i0}}|\geq c_0$ for $0\leq |\xi|<<1$ with $c_0$ being a certain positive constant. Then for $t$ sufficiently large, it holds
 \be
  \|(\u(t), \mathbf{n}(t))\|_{L^2}\geq C(1+t)^{-\frac{3}{4}}, \label{op}
 \ee
 i.e., the $L^2$-decay rates are optimal.
\el
\begin{proof}
We only have to consider the subsystem for $\mathbf{W}_i(t)$ ($i=1,2,3$) (recall the system \eqref{S1a}).
It follows from \eqref{S2a} and \eqref{G1}--\eqref{la2} that for sufficiently small $\eta>0$, it holds
\bea
&& |\widehat{u_i}(t,\xi)|\non\\
&=& \left|\widehat{u_{i0}}\left(\frac{-\mu|\xi|^2 (e^{\lambda_+ t}- e^{\lambda_- t})}{\lambda_+-\lambda_-}
  - \frac{ \lambda_- e^{\lambda_+ t}- \lambda_+e^{\lambda_- t}}{\lambda_+-\lambda_-} \right)
  +\widehat{n_{i0}}|\xi|
   \frac{e^{\lambda_+ t}- e^{\lambda_- t}}{\lambda_+-\lambda_-}\right|\non\\
   &\sim&
 \left\{\begin{array}{cc}
 O(1)e^{-\frac{\mu}{2}|\xi|^2t} (|\widehat{u_{i0}}|+ |\widehat{n_{i0}}|),& |\xi|<\eta,\\
 O(1) e ^{-\gamma t} |\widehat{u_{i0}}|+ O(1)|\xi|^{-1} e ^{-\gamma t}  |\widehat{n_{i0}}|, & \ |\xi|\geq \eta,
 \end{array} \right.\label{u1}
\eea
and
\bea
|\widehat{n_i}(t,\xi)|
 &=& \left| \widehat{u_{i0}}\frac{-|\xi| (e^{\lambda_+ t}- e^{\lambda_- t})}{\lambda_+-\lambda_-}
 - \widehat{n_{i0}}\frac{ \lambda_- e^{\lambda_+ t}- \lambda_+e^{\lambda_- t}}{\lambda_+-\lambda_-}\right|\non\\
   &\sim&
 \left\{\begin{array}{cc}
 O(1)e^{-\frac{\mu}{2}|\xi|^2t}( |\widehat{u_{i0}}|+|\widehat{n_{i0}}|),& |\xi|<\eta,\\
 O(1) e ^{-\gamma t} |\xi|^{-1} |\widehat{u_{i0}}|+ O(1) e ^{-\gamma t}|\widehat{n_{i0}}|, & \ |\xi|\geq \eta.
 \end{array} \right.
 \label{n1}
\eea

Since $\u_0, \mathbf{n}_0\in L^1$, their Fourier transforms lie in $L^\infty$, i.e., $|\widehat{\u_0}|_{L^\infty} \leq C$, $|\widehat{\mathbf{n}_0}|_{L^\infty} \leq C$. Then we have for $0\leq |\alpha|\leq l$,
\be
\begin{split}
\|\widehat{\partial_x^\alpha u_i}(t,\xi)\|_{L^2}^2
&= \int_{|\xi|< \eta}  |\widehat{\partial_x^\alpha  u_i}(t,\xi)|^2 d\xi+\int_{|\xi|\geq \eta} |\widehat{\partial_x^\alpha  u_i}(t,\xi)|^2 d\xi\\
&\leq C\int_{|\xi|< \eta}  e^{-\mu|\xi|^2t} |\xi|^{2|\alpha|}(|\widehat{ u_{i0}}|^2+ |\widehat{ n_{i0}}|^2) d\xi\\
&\quad\ +C\int_{|\xi|\geq \eta} (1+|\xi|^{-2}) e ^{-2\gamma t} |\xi|^{2|\alpha|} (|\widehat{u_{i0}}|^2+ |\widehat{ n_{i0}}|^2)d\xi\\
&\leq C(1+t)^{-\frac32-|\alpha|}\left(\|(u_{i0}, n_{i0})\|_{L^1}^2+\|(\partial_x^\alpha u_{i0}, \partial_x^\alpha n_{i0})\|_{L^2}^2\right).\non
\end{split}
\ee
Estimate on $\|\widehat{\partial_x^\alpha n_i}(t,\xi)\|_{L^2}$ can be obtained in a similar manner. Thus, \eqref{decun} is proved.

Next, it follows from \eqref{la1}, \eqref{la2}  and \eqref{u1}  that
\be
\begin{split}
\|u_i(t,x)\|_{L^2}^2&=\|\widehat{u_i}(t,\xi)\|_{L^2}^2\\
&= \int_{|\xi|< \eta}  |\widehat{ n_i}(t,\xi)|^2 d\xi+\int_{|\xi|\geq \eta} |\widehat{  n_i}(t,\xi)|^2 d\xi\\
& \geq
 C\int_{|\xi|< \eta}  e^{-\mu|\xi|^2t}\Big[-\Big( \mu|\xi|^2\frac{\sin bt}{b}\Big)^2+ (\cos bt )^2 \Big] |\widehat{ u_{i0}}|^2
 d\xi\\
 &\quad\
 -C\int_{|\xi|< \eta}  e^{-\mu|\xi|^2t}|\xi|^2\Big(\frac{\sin bt}{b}\Big)^2 |\widehat{ n_{i0}}|^2 d\xi\\
&\quad\ -C e^{-2\gamma t}\int_{|\xi|\geq \eta} \left(|\widehat{ u_{i0}}|^2+ |\widehat{ n_{i0}}|^2\right) d\xi\\
&\geq C\int_{|\xi|< \eta} e^{-\mu|\xi|^2t}(\cos bt)^2  |\widehat{ u_{i0}}|^2 d\xi\\
&\quad\ -C\int_{|\xi|< \eta}  e^{-\mu|\xi|^2t} \frac{\mu^2|\xi|^4+|\xi|^2}{b^2} (\sin bt)^2 ( |\widehat{ u_{i0}}|^2+ |\widehat{ n_{i0}}|^2) d\xi\\
&\quad\ -C e^{-2\gamma t}\\
&:= CI_1-CI_2-C e^{-2\gamma t}.
\end{split}\label{op1}
\ee
On the other hand, we infer from \eqref{b} and the mean value formula that
\be
\begin{split}
 (\cos bt)^2 &= \cos^2(|\xi|t+O(|\xi|^3)t)\\
 &\geq  \frac12 \cos^2 (|\xi|t)- O(|\xi|^6t^2), \quad |\xi|<\eta,\label{cos}
 \end{split}
 \ee
 \be
 \begin{split}
 (\sin bt)^2 &= \sin^2(|\xi|t+O(|\xi|^3)t) \\
 &\leq  \sin^2 (|\xi|t)+ O(|\xi|^6t^2), \quad |\xi|<\eta. \label{sin}
 \end{split}
 \ee
Then using \eqref{cos}, we can estimate $I_1$ as follows
\be
\begin{split}
 I_1&=\int_{|\xi|< \eta} e^{-\mu|\xi|^2t}(\cos bt)^2  |\widehat{ u_{i0}}|^2 d\xi\\
&\geq \frac12 \int_{|\xi|< \eta} e^{-\mu|\xi|^2t} \cos^2 (|\xi|t)|\widehat{ u_{i0}}|^2 d\xi\\
&\quad \ -\int_{|\xi|< \eta} e^{-\mu|\xi|^2t}O(|\xi|^6t^2) |\widehat{ u_{i0}}|^2 d\xi.\label{I1}
\end{split}
\ee
For $t\geq t_0:= 2\pi\eta^{-1}$ sufficiently large, using the assumption $|\widehat{u_{i0}}| \geq c_0>0$, we get
\be
\begin{split}
&\int_{|\xi|< \eta} e^{-\mu|\xi|^2t} \cos^2 (|\xi|t)|\widehat{ u_{i0}}|^2 d\xi\\
&\ = C\int_0^{\eta} e^{-\mu|\xi|^2t} \cos^2 (|\xi|t)|\widehat{ u_{i0}}|^2 |\xi|^2 d |\xi|\\
&\ = Ct^{-\frac32}\int_0^{\eta t^\frac12} e^{-\mu|\zeta|^2} \cos^2 (|\zeta|t^\frac12)|\widehat{ u_{i0}}|^2 |\zeta|^2 d |\zeta|\\
&\ \geq Cc_0^2(1+t)^{-\frac32}\int_0^{\eta t^\frac12} e^{-\mu r^2} \cos^2 (r t^\frac12) r^2 dr\\
&\ \geq  Cc_0^2(1+t)^{-\frac32}\sum_{k=0}^{[\eta t \pi^{-1}]-1}\int_{k\pi t^{-\frac12}}^{(k\pi+\frac14\pi)t^{-\frac12}}
e^{-\mu r^2} \cos^2 \Big( k\pi+ \frac{\pi}{4}\Big) r^2 dr\\
&\ \geq  Cc_0^2(1+t)^{-\frac32}.\label{op3}
\end{split}
\ee
Besides, we have
\be
\int_{|\xi|< \eta} e^{-\mu|\xi|^2t}|\xi|^6t^2 |\widehat{ u_{i0}}|^2 d\xi\leq C(1+t)^{-\frac52},\label{op2}
\ee
which together with \eqref{I1} and \eqref{op3} implies that $$I_1\geq Cc_0^2 (1+t)^{-\frac32}, \quad \forall\, t>>1.$$
 Next, it follows from \eqref{b} and \eqref{sin} that
\be
\begin{split}
I_2&\leq C\int_{|\xi|<\eta}e^{-2\mu|\xi|^2t}\left(1+|\xi|^2\right)\left(|\xi|^2t^2+|\xi|^6t^2) (|\widehat{ u_{i0}}|^2+ |\widehat{ n_{i0}}|^2\right) d\xi\\
&\leq C(1+t)^{-\frac52}.\label{op4}
\end{split}
\ee
Hence, we can conclude from \eqref{op1} and the above estimates for $I_1, I_2$ that for sufficiently large $t$, it holds
\be
\|u_i(t,x)\|_{L^2}\geq Cc_0 (1+t)^{-\frac34}.\non
\ee
Using a similar argument, we can show the lower bound of time decay rate for $\mathbf{n}$ such that
 \be
 \|n_i(t,x)\|_{L^2}\geq Cc_0 (1+t)^{-\frac34}, \quad \forall\, t>>1. \non
 \ee
 Thus the estimate \eqref{op} follows. The proof is complete.
\end{proof}
We remark that the linearized system of \eqref{OE} for variables $(\Omega, \mathbb{E})$ has the following form:
\be
\label{LOE}
\begin{cases}
& \Omega_t-\mu\Delta \Omega-\Lambda \mathbb{E}=0,\\
& \mathbb{E}_t+\Lambda \Omega= 0.
\end{cases}
\ee
Denote
$
\mathbf{V}_{ij}=( \Omega_{ij}, \mathbb{E}_{ij})^T$, $ i, j=1,2,3$, $i\neq j$.
The linear problem  \eqref{LOE} can be written as follows
\be
\begin{cases}
&\mathbf{V}_{ijt}= \mathcal{B}\mathbf{V}_{ij},\\
& \mathbf{V}_{ij}(t,x)|_{t=0}=(  \Omega_{ij0}, \mathbb{E}_{ij0})^T,
\end{cases}
 \label{S1b}
\ee
where $\mathcal{B}$ is given by \eqref{B}.

As a consequence, we can obtain decay estimates that are exactly the same as those in Lemma \ref{linear1} for the linear problem \eqref{LOE}:
\bl\label{linear2}
 Let $(\Omega(t), \mathbb{E}(t))$ be a solution to the linear system \eqref{LOE} with initial data $(\Omega_0, \mathbb{E}_0)\in H^l(\mathbb{R}^3)\cap L^1 (\mathbb{R}^3)$. Then we have for $0\leq |\alpha|\leq l$ and $t\geq 0$ such that
 \be
 \begin{split}
 & \|(\partial_x^\alpha \Omega(t), \partial_x^\alpha \mathbb{E}(t))\|_{L^2}\\
 &\ \leq C(1+t)^{-\frac34-\frac{|\alpha|}{2}}(\|(\Omega_0, \mathbb{E}_0)\|_{L^1}+\|\partial_x^\alpha (\Omega_0, \mathbb{E}_0)\|_{L^2}).\non
 \end{split}
 \ee
 Moreover, assume that $\widehat{\Omega_0}$ and $\widehat{\mathbb{E}_0}$ satisfies $|\widehat{\Omega_0}|\geq c_0$, $|\widehat{\mathbb{E}_0}|\geq c_0$ for $0\leq |\xi|<<1$ with $c_0$ being a positive constant. Then for $t$ sufficiently large, it holds
 \be
 \|(\Omega(t), \mathbb{E}(t))\|_{L^2}\geq C(1+t)^{-\frac{3}{4}}. \non
 \ee
\el

\subsection{Decay estimates for the nonlinear system}

Now we proceed to prove  Theorem \ref{main}. The proof consists of several steps. Below we always assume that $(\u, \E)$ is the global smooth solution to \eqref{e1} obtained in Proposition \ref{glo} with sufficiently small $\delta$.

\bl\label{prea}
Under the assumptions of Proposition \ref{glo}, we have the following estimates on the nonlinear terms $\mathbf{g}$ and $\mathbf{h}$:
\be
  \|\widehat{\mathbf{g}}\|_{L^\infty_\xi}
  \leq C\|\u\|_{L^2}\|\nabla \u \|_{L^2}+C\|\E\|_{L^2}\|\nabla \E\|_{L^2},\label{egg1}
 \ee
 \be
  \|\widehat{\Lambda^{-1}\nabla \cdot \mathbf{h}}\|_{L^\infty_\xi}
  \leq  C\|\u\|_{L^2}\|\nabla \E \|_{L^2}+C\|\E\|_{L^2}\|\nabla \u\|_{L^2}.
 \label{ehh1}
\ee
\el
\begin{proof}

We infer from the expression \eqref{gh}  that
\be
\widehat{\mathbf{g}}(\xi)= \Big(\mathbb{I}-\frac{\xi\otimes\xi}{|\xi|^2}\Big)\Big[-\widehat{\u\cdot \nabla \u}(\xi)+\widehat{\nabla \cdot(\E\E^T)}(\xi)\Big].\non
\ee
As a result,
\be
\begin{split}
 \|\widehat{\mathbf{g}}(\xi)\|_{L^\infty_\xi}
 &\leq  C\|\widehat{\u\cdot \nabla \u}(\xi)\|_{L^\infty_\xi}+\|\widehat{\nabla \cdot(\E\E^T)}(\xi)\|_{L^\infty_\xi}\\
&\leq C\|\u\cdot\nabla \u\|_{L^1}+ C\|\nabla \cdot(\E\E^T)\|_{L^1}\\
&\leq C\|\u\|_{L^2}\|\nabla \u \|_{L^2}+C\|\E\|_{L^2}\|\nabla \E\|_{L^2},\non
\end{split}
\ee
which yields the conclusion \eqref{egg1}. The estimate \eqref{ehh1} can be obtained in a similar way. The proof is complete.
\end{proof}

Now we introduce the following functions (also compare with \eqref{Ge})
\bea && H(t)=\frac12(\|\Delta \u\|_{L^2}^2+\|\Delta \E\|_{L^2}^2)+\kappa (\Lambda \Omega, \Delta \mathbb{E}), \non\\
&& \widetilde{H}(t)=\sup_{0\leq s\leq t} (1+s)^{\frac52}\Big( H(t)+\frac12\|\nabla \u\|_{L^2}^2+\frac12\|\nabla \E\|_{L^2}^2\Big),\non
\eea
where $\kappa>0$ is a sufficiently small constant such that
\be
\begin{split}
\|\nabla  \u\|_{H^1}^2+\|\nabla  \E\|_{H^1}^2
&\geq \frac12\|\nabla \u\|_{L^2}^2 +\frac12\|\nabla \E\|_{L^2}^2+ H(t)\\
&\geq \frac14\left(\|\nabla \u\|_{H^1}^2+\|\nabla \E\|_{H^1}^2\right).
\label{HH}
\end{split}
\ee

Then we have the following estimates
\bl\label{h1u}
Under the assumptions of Proposition \ref{glo}, for any $t\geq 0$, the global solution $(\u, \E)$ satisfies
\be
\begin{split}
& \| \u (t)\|_{L^2}+\| \E(t) \|_{L^2}\\
&\ \ \leq C(1+t)^{-\frac34}\left(\|\u_0\|_{L^1\cap L^2}+ \|\E_0\|_{L^1\cap L^2}+ \delta \widetilde{H}^\frac12(t)\right),\label{L2uE}
\end{split}
\ee
\be
\begin{split}
& \|\nabla \u (t)\|_{L^2}+\|\nabla \E(t) \|_{L^2}\\
&\ \ \leq C(1+t)^{-\frac54}\left(\|\u_0\|_{L^1\cap H^1}+ \|\E_0\|_{L^1\cap H^1}+ \delta \widetilde{H}^\frac12(t)\right).\label{H1uE}
\end{split}
\ee
\el
\begin{proof}
It follows from Lemma \ref{pre} that if $\delta$ is sufficiently small, then we have the following equivalent relations
\be
\|\E\|_{L^2}\approx\|\mathbf{n}\|_{L^2},\quad \|\nabla \E\|_{L^2}\approx\|\nabla \mathbf{n}\|_{L^2},\label{eqEn}
\ee
Recalling \eqref{un}, we only have to consider the nonlinear system for the vector $\mathbf{W}_i=(u_i, n_i)^T$ ($i=1,2,3$)
\be
\begin{cases}
& \mathbf{W}_{it}=\mathcal{B}\mathbf{W}_i+ \mathbf{f}_i,\\
& \mathbf{W}_i(t,x)|_{t=0}=\mathbf{W}_{i0}:=(u_{i0}, \Lambda^{-1}(\nabla \cdot \E)_{i0})^T,
\end{cases}\non
\ee
where $$\mathbf{f}_i=(g_i, \Lambda^{-1}(\nabla \cdot \mathbf{h})_i)^T$$
(see \eqref{gh} for the definition of $\mathbf{g}$ and $\mathbf{h}$).

From the Duhamel's principle, we have
\be
\mathbf{W}_i(t)=e^{t\mathcal{B}} \mathbf{W}_{i0}+\int_0^t e^{(t-s)\mathcal{B}} \mathbf{f}_i(s) ds.\non
\ee
It follows from the linear decay estimate \eqref{decun} (see Lemma \ref{linear1}) that
\be
\begin{split}
 \|\mathbf{W}_i(t)\|_{L^2} &\leq  C (1+t)^{-\frac34}\|\mathbf{W}_{i0}\|_{L^1\cap L^2}\\
 &\quad \ +C \int_0^t (1+t-s)^{-\frac34}\left(\|\widehat{\mathbf{f}_i}(s)\|_{L^\infty_\xi}+\|\mathbf{f}_i(s)\|_{L^2}\right)ds, \non
 \end{split}
 \ee
 \be
 \begin{split}
 \|\nabla \mathbf{W}_i(t)\|_{L^2}
 & \leq C (1+t)^{-\frac54}\|\mathbf{W}_{i0}\|_{L^1\cap H^1}\\
 &\quad \ +C \int_0^t (1+t-s)^{-\frac54}\left(\|\widehat{\mathbf{f}_i}(s)\|_{L^\infty_\xi}+\|\nabla \mathbf{f}_i(s)\|_{L^2}\right)ds, \non
 \end{split}
\ee
which imply that
\be
\begin{split}
& \|( \u(t), \mathbf{n}(t))\|_{L^2}\\
&\leq C (1+t)^{-\frac34}\|(\u_0, \E_0)\|_{L^1\cap L^2}\\
&\quad  +C \int_0^t (1+t-s)^{-\frac34}\left(\|(\widehat{\mathbf{g}}, \widehat{\Lambda^{-1}\nabla \cdot \mathbf{h}})(s)\|_{L^\infty_\xi}+\| (\mathbf{g}, \Lambda^{-1}\nabla \cdot \mathbf{h})(s)\|_{L^2}\right)ds,\label{nes0}
\end{split}
\ee
and
\be
\begin{split}
& \|(\nabla \u(t), \nabla \mathbf{n}(t))\|_{L^2}\\
&\leq C (1+t)^{-\frac54}\|(\u_0, \E_0)\|_{L^1\cap H^1}\\
&\quad\ +C \int_0^t (1+t-s)^{-\frac54}\left(\|(\widehat{\mathbf{g}}, \widehat{\Lambda^{-1}\nabla \cdot \mathbf{h}})(s)\|_{L^\infty_\xi}+\|\nabla (\mathbf{g}, \Lambda^{-1}\nabla \cdot \mathbf{h})(s)\|_{L^2}\right)ds.\label{nes}
\end{split}
\ee
By the estimates \eqref{stability}, \eqref{nes0},  Lemmas \ref{pre1}, \ref{prea} and the definition of $\widetilde{H}(t)$, we have
\be
\begin{split}
& \|(\u(t), \mathbf{n}(t))\|_{L^2}\\
&\leq C (1+t)^{-\frac34}\|(\u_0, \E_0)\|_{L^1\cap L^2}\\
&\quad\ + C\delta \int_0^t (1+t-s)^{-\frac34}\left(\|\Delta \u(s)\|_{L^2}+\|\nabla \u(s)\|_{L^2}\right)ds\\
&\quad\ + C\delta \int_0^t (1+t-s)^{-\frac34}\left(\|\Delta \E(s)\|_{L^2}+\|\nabla \E(s)\|_{L^2}\right)ds\\
&\leq C (1+t)^{-\frac34}\|(\u_0, \E_0)\|_{L^1\cap L^2} +C\delta \int_0^t (1+t-s)^{-\frac34}(1+s)^{-\frac54}\widetilde{H}^\frac12(s) ds\\
&\leq C (1+t)^{-\frac34}\|(\u_0, \E_0)\|_{L^1\cap L^2} +C\delta \widetilde{H}^\frac12(t)\int_0^t (1+t-s)^{-\frac34}(1+s)^{-\frac54} ds\\
&\leq C (1+t)^{-\frac34}\left(\|(\u_0, \E_0)\|_{L^1\cap L^2}+\delta \widetilde{H}^\frac12(t)\right),\label{H1unaa}
\end{split}
\ee
where in the last step we have used the following elementary inequality (see e.g., \cite[Lemma 5.2.7]{Z04})
$$ \int_0^t(1+t-s)^{-\gamma}(1+s)^{-\beta} ds\leq C(1+t)^{-\gamma}, \quad \forall\, t\geq 0,$$
for $\beta>1$, $\beta\geq \gamma\geq 0$.

Similarly, we can deduce from \eqref{nes} that
\be
\begin{split}
& \|(\nabla \u(t), \nabla \mathbf{n}(t))\|_{L^2}\\
&\ \leq C (1+t)^{-\frac54}\|(\u_0, \E_0)\|_{L^1\cap H^1} +C\delta \int_0^t (1+t-s)^{-\frac54}(1+s)^{-\frac54}\widetilde{H}^\frac12(s) ds\\
&\ \leq C (1+t)^{-\frac54}\left(\|(\u_0, \E_0)\|_{L^1\cap H^1}+\delta \widetilde{H}^\frac12(t)\right),\label{H1una}
\end{split}
\ee
Then we can conclude \eqref{L2uE} and \eqref{H1uE} from the estimates \eqref{H1unaa}, \eqref{H1una} and  the relation \eqref{eqEn}. The proof is complete.
\end{proof}
Now we are in a position to prove the $L^2$-decay rate estimate.
\begin{proposition}\label{decayH2}
Under the assumptions of Proposition \ref{glo}, if $\delta>0$ is sufficiently small, then for $t\geq 0$, the global solution $(\u, \E)$ to system \eqref{e1} satisfies the following decay estimates:
\be
 \|\u (t)\|_{L^2}+\| \E(t) \|_{L^2}\leq CM(1+t)^{-\frac34},\label{esL2}
 \ee
 \be
 \|\nabla \u (t)\|_{H^1}+\|\nabla  \E(t) \|_{H^1} \leq CM(1+t)^{-\frac54}.\label{esH2}
\ee
Moreover, we have the $L^p$ decay rate
\be
\label{esLp}
 \|\u (t)\|_{L^p}+\| \E(t) \|_{L^p}\leq
  \begin{cases}
  &CM(1+t)^{-\frac32\big(1-\frac1p\big)}, \quad \ p\in[2,6],\\
  &CM(1+t)^{-\frac54}, \qquad \quad\ \  p\in [6,\infty],
  \end{cases}
\ee
where $M=\|\u_0\|_{L^1\cap H^2}+ \|\E_0\|_{L^1\cap H^2}$.
\end{proposition}
\begin{proof}
We infer from Lemma \ref{diin}, \eqref{stability} (see Proposition \ref{glo}) and in particular the relation \eqref{EEE} that
\be
\begin{split}
& \frac{d}{dt}H(t)+ (\mu-C\kappa)\|\nabla \Delta \u\|_{L^2}^2+\frac12\|\Delta \mathbb{E}\|_{L^2}^2\\
&\ \leq C\kappa\|\nabla\u\|_{L^2}^2+C(\delta+\delta^2) \left(\|\Delta \E\|_{L^2}^2+\|\nabla \u\|_{L^2}^2+\|\nabla \Delta \u\|_{L^2}^2\right)\\
&\ \leq C\kappa\|\nabla\u\|_{L^2}^2+ C(\delta+\delta^2)\Big( \frac{C}{1-C\delta}\|\Delta \mathbb{E}\|_{L^2}^2+\|\nabla \u\|_{L^2}^2+\|\nabla \Delta \u\|_{L^2}^2\Big).\label{kkk}
\end{split}
\ee
As a result, if $\kappa, \delta>0$ are sufficiently small, we infer from \eqref{kkk}, the interpolation inequality $
 \|\Delta \u\|_{L^2}^2\leq \|\nabla \u\|_{L^2}\|\nabla \Delta \u\|_{L^2}$ and the relation \eqref{EEE} that there exist constants  $\kappa_1, C$ such that
\be
\frac{d}{dt}H(t)+\kappa_1 H(t)\leq C\|\nabla \u\|_{L^2}^2.\label{HHH}
\ee
Applying the Gronwall inequality to \eqref{HHH} and using \eqref{H1uE} (see Lemma \ref{h1u}), we have
\be
\begin{split}
& H(t)+\left(\frac12\|\nabla \u(t)\|_{L^2}^2+\frac12\|\nabla \E(t)\|_{L^2}^2\right)\\
&\ \leq H(0)e^{-\kappa_1 t} + C e^{-\kappa_1 t}\int_0^t e^{\kappa_1 s}\|\nabla \u(s)\|_{L^2}^2ds\\
&\quad\ +\left(\frac12\|\nabla \u(t)\|_{L^2}^2+\frac12\|\nabla \E(t)\|_{L^2}^2\right)\\
&\ \leq H(0)e^{-\kappa_1 t} \\
&\quad \ + C e^{-\kappa_1 t}\int_0^t e^{\kappa_1 s}(1+s)^{-\frac52}\Big(\|\u_0\|_{L^1\cap H^1}^2+ \|\E_0\|_{L^1\cap H^1}^2+ \delta^2 \widetilde{H}(s)\Big)dx\\
&\quad \ + C(1+t)^{-\frac52}\left(\|\u_0\|_{L^1\cap H^1}^2+ \|\E_0\|_{L^1\cap H^1}^2+ \delta^2 \widetilde{H}(t)\right)\\
&\ \leq H(0)e^{-\kappa_1 t}\\
&\quad \ +C\left(\|\u_0\|_{L^1\cap H^1}^2+ \|\E_0\|_{L^1\cap H^1}^2+ \delta^2 \widetilde{H}(t)\right)\int_0^{t}e^{-\kappa_1 (t-s)}(1+s)^{-\frac52}ds\\
&\quad \  + C(1+t)^{-\frac52}\left(\|\u_0\|_{L^1\cap H^1}^2+ \|\E_0\|_{L^1\cap H^1}^2+ \delta^2 \widetilde{H}(t)\right).\label{HHHa}
\end{split}
\ee
Since
\be
\begin{split}
& \int_0^{t}e^{-\kappa_1 (t-s)}(1+s)^{-\frac52}ds \\
&\ \leq \int_0^{\frac{t}{2}}e^{-\kappa_1 (t-s)}(1+s)^{-\frac52}ds+\int_{\frac{t}{2}}^te^{-\kappa_1 (t-s)}(1+s)^{-\frac52}ds\\
&\ \leq e^{-\frac{\kappa_1 t}{2}}\int_0^{\frac{t}{2}}(1+s)^{-\frac52}ds+\Big(1+\frac{t}{2}\Big)^{-\frac52}\int_{\frac{t}{2}}^te^{-\kappa_1 (t-s)}ds\\
&\ \leq C(1+t)^{-\frac52},\non
\end{split}
\ee
we deduce from \eqref{HHHa} that
\be
\widetilde{H}(t)\leq  H(0)e^{-\kappa_1 t}(1+t)^{\frac52} +C\Big(\|\u_0\|_{L^1\cap H^1}^2+ \|\E_0\|_{L^1\cap H^1}^2+ \delta^2 \widetilde{H}(t)\Big).\non
\ee
If $\delta>0$ is sufficiently small, we get for $t\geq 0$,
\be
\begin{split}
\widetilde{H}(t) &\leq
 C\left(H(0)+\|\u_0\|_{L^1\cap H^1}^2+ \|\E_0\|_{L^1\cap H^1}^2\right)\\
 &\leq C\left(\|\u_0\|_{L^1\cap H^2}^2+ \|\E_0\|_{L^1\cap H^2}^2\right).
 \label{esHH}
 \end{split}
\ee
The uniform estimate \eqref{esHH} together with \eqref{HH}--\eqref{H1uE} yields the decay estimates \eqref{esL2} and \eqref{esH2}.

Next, by the Sobolev embedding theorem and the Gagliardo-Nirenberg inequality, we obtain that
\be
\begin{split}
\|\u(t)\|_{L^6}+\|\mathbf{n}(t)\|_{L^6}&\leq C(\|\nabla \u(t)\|_{L^2}+\|\nabla \mathbf{n}(t)\|_{L^2})\\
&\leq C(1+t)^{-\frac54},\non
\end{split}
\ee
\be
\begin{split}
&\|\u(t)\|_{L^\infty}+\|\mathbf{n}(t)\|_{L^\infty}\\
&\ \leq C\left(\|\nabla \u(t)\|_{L^2}^\frac12\|\Delta \u(t)\|_{L^2}^\frac12+\|\nabla \mathbf{n}(t)\|_{L^2}^\frac12\|\Delta \mathbf{n}(t)\|_{L^2}^\frac12\right)\\
&\ \leq C(1+t)^{-\frac54}.\non
\end{split}
\ee
Then we can conclude the $L^p$ decay rate estimate \eqref{esLp} by interpolation. The proof is complete.
\end{proof}

Finally, we show that the $L^2$-decay rate \eqref{esL2} is optimal.

\begin{proposition}\label{nonop}
Suppose that the assumptions of Proposition \ref{glo} are satisfied. $\delta$ is a sufficiently small positive constant and the initial data satisfies $ \|\u_0\|_{H^2}+\|\E_0\|_{H^2}\leq \delta$. Moreover, we assume that $(\u_0, \mathbf{n}_0)$ $(\text{with}\ \mathbf{n}_0=\Lambda^{-1} \nabla \cdot \E_0)$ also satisfy
$$
 |\widehat{u_{i0}}|\geq c_0, \quad |\widehat{n_{i0}}|\geq c_0, \quad \text{for}\ \ 0\leq |\xi|<<1,$$
  where $c_0$ satisfies $c_0\sim O(\delta^\zeta)$  with $\zeta\in (0,1)$. Then the global solution $(\u, \E)$ to system \eqref{e1} has the following estimate
\be
\| \u (t)\|_{L^2}+\| \E(t) \|_{L^2}\geq C(1+t)^{-\frac34}, \quad \forall\,t\geq t_0,\label{opL2}
\ee
where $t_0$ is sufficiently large.
\end{proposition}
\begin{proof}
By the lower bound of $L^2$ decay \eqref{op} (see Lemma \ref{linear1} (ii)), the Duhamel's principle, \eqref{egL12}, \eqref{ehL12}, the definition of $\widetilde{H}(t)$ and the estimates \eqref{stability}, \eqref{esHH}, we deduce that
\be
\begin{split}
\|(\u(t), \mathbf{n}(t))\|_{L^2}
&\geq Cc_0 (1+t)^{-\frac34}\\
&\quad\ - C\sup_{0\leq s\leq t}(\|\u(s)\|_{H^1}+\|\E(s)\|_{H^1}) \\
&\qquad\ \ \times \int_0^t (1+t-s)^{-\frac34}(\|\Delta \u(s)\|_{L^2}+\|\Delta \E(s)\|_{L^2})ds\\
&\quad\ - C\sup_{0\leq s\leq t}(\|\u(s)\|_{L^2}+\|\E(s)\|_{L^2}) \\
&\qquad\ \ \times \int_0^t (1+t-s)^{-\frac34}(\|\nabla \u(s)\|_{L^2}+\|\nabla \E(s)\|_{L^2})ds\\
&\geq C\delta^\eta (1+t)^{-\frac34}-C\delta \widetilde{H}^\frac12(t)\int_0^t (1+t-s)^{-\frac34}(1+s)^{-\frac54} ds\\
&\geq C(\delta^\eta-C\delta) (1+t)^{-\frac34}\\
&\geq C(1+t)^{-\frac34},\non
\end{split}
\ee
provided that $\delta $ (and thus $c_0$) is sufficiently small.

Combining the above estimate with \eqref{eqEn}, we can conclude \eqref{opL2}. The proof is complete.
\end{proof}

\textbf{Proof of Theorem \ref{main}}. The conclusions of Theorem \ref{main} now easily follow from Propositions \ref{glo}, \ref{decayH2} and \ref{nonop}. The proof is complete.

\section{Weak-strong uniqueness}
\setcounter{equation}{0}
\noindent

In this section, we proceed to prove Theorem \ref{unique}.
 The pair $(\u,\E)$ will denote a classical solution (for instance, the global classical solution obtained in Proposition \ref{glo}, the proof for local classical solution in \cite[Theorem 1]{LLZ08} will be the same)
and the pair $(\hat{\u},\hat{\E})$ will denote a finite energy weak solution to the incompressible viscoelastic system \eqref{e1} as in Definition \ref{defw}.

Suppose that $(\u,\E)$ and $(\hat{\u},\hat{\E})$ share the same initial data $(\u_0,\E_0)$ satisfying
$\u_0, \E_0 \in H^k(\mathbb{R}^d)$ $(k\geq 3)$ and the assumptions (A1)--(A3). Our aim is to prove that $(\hat{\u},\hat{\E})=(\u,\E)$.

For this purpose, we
shall estimate the differences of solutions denoted by
$$\mathfrak{U}=\hat{\u}-\u\quad\textrm{and}\quad \mathfrak{E}=\hat{\E}-\E.$$
 In the variational formulation \eqref{w1} for the weak solution $\hat{\u}$, taking $\mathbf{v}=\u$ as the test function (which makes sense as $\u$ is regular)\footnote{Actually one needs to use smooth functions as
the test functions according to definitions in \eqref{w1}--\eqref{w2}, but this can be done via a standard regularization procedure in view of the regularity of $\u$ and $\E$.} and using the equation for $\u$ (see \eqref{e1}),
we obtain that
 \be
 \begin{split}
&\int_{\R^d}\hat{\u}(\cdot, t)\cdot\mathbf{u}(\cdot, t)dx -\int_{\R^d} \hat{\u}_0\cdot \u_0 dx\\
&\ = \int_0^t\int_{\R^d}\left[\hat{\u}\cdot\mathbf{u}_t+\left(\hat{\u}\otimes\hat{\u}-\hat{\E}\hat{\E}^T-\hat{\E}\right):\nabla\mathbf{u}\right] dxd\tau\\
&\quad \ -\mu\int_0^t\int_{\R^d}\nabla\hat{\u}:\nabla\mathbf{u} dxd\tau\\
&\ =\int_0^t\int_{\R^d}\left[-\left(\u\cdot \nabla \u\right)\cdot \hat{\u}+ \left(\hat{\u}\otimes\hat{\u}\right): \nabla\mathbf{u} \right]dx d\tau\\
&\quad\ +\int_0^t\int_{\R^d}\left[ \left(\nabla \cdot\left(\E+\E\E^T\right)\right)\cdot \hat{\u}- (\hat{\E}+\hat{\E}\hat{\E}^T):\nabla \u \right]dxd\tau\\
&\quad\ -2\mu\int_0^t\int_{\R^d}\nabla\hat{\u}:\nabla\mathbf{u} dxd\tau.
\label{w4}
\end{split}
\ee
Next, we test \eqref{w2} by $\E_j$ and sum up with respect to $j$. Using the equation \eqref{Ej} for $\E$, we get
\bea
&&  \int_{\R^d}\hat{\E} : \E dx-\int_{\R^d} \hat{\E}_{0} : \E_{0} dx\non\\
&=& \sum_{j=1}^d \int_0^t\int_{\R^d}\Big[\hat{\E}_j\cdot \E_{jt}+\left(\hat{\E}_j\otimes\hat{\u}-\hat{\u}\otimes\hat{\E}_j\right):\nabla \E_j -\hat{\u}\cdot \nabla _j \E_j\Big]dxd\tau\non\\
&=&
\sum_{j=1}^d \int_0^t\int_{\R^d}\Big[ -\nabla \cdot (\E_j\otimes\u-\u\otimes\E_j)\cdot \hat{\E}_j+\left(\hat{\E}_j\otimes\hat{\u}-\hat{\u}\otimes\hat{\E}_j\right):\nabla \E_j \Big]dxd\tau\non\\
&& + \int_0^t\int_{\R^d} (\nabla \u: \hat{\E}- \hat{\u}\cdot (\nabla\cdot \E)) dxd\tau.
\label{w5}
\eea

Recall that the weak solution $(\hat{\u}, \hat{\E})$ satisfies the energy inequality \eqref{w3} while the strong solution $(\u, \E)$ satisfies the energy equality such that for $t\in [0, T]$, we have
\be
  \frac12\left(\|\hat{\u}(t)\|_{L^2}^2+\|\hat{\E}(t)\|_{L^2}^2\right)
     + \mu\int_0^t\|\nabla \hat{\u}(\tau)\|_{L^2}^2 d\tau
     \leq \frac12\left(\|\hat{\u}_0\|_{L^2}^2+\|\hat{\E}_0\|_{L^2}^2\right),\label {w6}
\ee
and
\be
 \frac12\left(\|\u(t)\|_{L^2}^2+\|\E(t)\|_{L^2}^2\right)
     + \mu\int_0^t\|\nabla \u(\tau)\|_{L^2}^2 d\tau
     =\frac12\left(\|\u_0\|_{L^2}^2+\|\E_0\|_{L^2}^2\right).\label{w7}
\ee
Then it follows from \eqref{w4}--\eqref{w7} that
\be
 \frac12\left(\|\mathfrak{U}\|_{L^2}^2+\|\mathfrak{E}\|_{L^2}^2\right)
 +\mu\int_0^t  \|\nabla \mathfrak{U}\|^2_{L^2} d\tau
\leq \frac12\left(\|\mathfrak{U}_0\|_{L^2}^2+\|\mathfrak{E}_0\|_{L^2}^2\right) + R_1+R_2,\label{w8}
\ee
where the two reminder terms $R_1$ and $R_2$ are given by
\be
R_1= \int_0^t\int_{\R^d}\Big[(\u\cdot \nabla \u)\cdot \hat{\u}- (\hat{\u}\otimes\hat{\u}): \nabla\mathbf{u} \Big]dx d\tau,\non
\ee
\be
\begin{split}
R_2&= \int_0^t\int_{\R^d}\Big(-[\nabla \cdot(\E\E^T)]\cdot \hat{\u} +\sum_{j=1}^d [\nabla \cdot (\E_j\otimes\u-\u\otimes\E_j)]\cdot \hat{\E}_j\Big) dxd\tau\\
&\quad +\int_0^t\int_{\R^d} \Big[(\hat{\E}\hat{\E}^T):\nabla \u -\sum_{j=1}^d\left(\hat{\E}_j\otimes\hat{\u}-\hat{\u}\otimes\hat{\E}_j\right):\nabla \E_j\Big]dxd\tau.\non
\end{split}
\ee
Using the incompressible condition, we compute that
\be
\begin{split}
R_1&= -\int_0^t\int_{\R^d}\Big[(\u\cdot \nabla \u)\cdot \hat{\u}-(\u\cdot\nabla \u)\cdot \u-(\hat{\u}\cdot \nabla \u)\cdot \hat{\u}+(\hat{\u}\cdot\nabla \u)\cdot \u\Big] dxd\tau\\
&= -\int_0^t\int_{\R^d} (\mathfrak{U}\cdot\nabla \u)\cdot \mathfrak{U}dx d\tau\\
&\leq C\int_0^t \|\nabla \u\|_{L^\infty}\|\mathfrak{U}\|_{L^2}^2 d\tau.\non
\end{split}
\ee
Besides, we deduce from Lemma \ref{trans} and the fact $\nabla \cdot(\E \E^T)=\sum_{j=1}^{d} (\E_j\cdot \nabla) \E_j$ that
\be
\begin{split}
R_2
&=-\sum_{j=1}^d \int_0^t\int_{\R^d} (\E_j\cdot \nabla) \E_j \cdot \hat{\u} dxd\tau +\sum_{j=1}^d \int_0^t \int_{\R^d} (\u\cdot \nabla) \E_j\cdot \hat{\E}_jdx d\tau\\
&\quad\ -\sum_{j=1}^d  \int_0^t\int_{\R^d} (\E_j\cdot \nabla) \u\cdot  \hat{\E}_j  dxd\tau
 +\int_0^t\int_{\R^d} (\hat{\E}\hat{\E}^T):\nabla \u dx d\tau\\
&\quad\
 -\sum_{j=1}^d \int_0^t \int_{\R^d} (\hat{\u}\cdot \nabla)\E_j\cdot \hat{\E}_j dx d\tau + \sum_{j=1}^d \int_0^t\int_{\R^d} (\hat{\E}_j\cdot \nabla )\E_j \cdot \hat{\u} dxd\tau\\
&:= I_1+...+I_6.\non
\end{split}
\ee
Using the incompressibility condition again, after a tedious computation, we get
\be
\begin{split}
I_2+I_5 &= -\sum_{j=1}^d  \int_0^t\int_{\R^d} (\mathfrak{U} \cdot  \nabla)  \E_j\cdot \mathfrak{E}_j dxd\tau\\
 &\leq C \int_0^t\|\nabla \E\|_{L^\infty}\|\mathfrak{U}\|_{L^2}\|\mathfrak{E}\|_{L^2}d\tau \\
 &\leq C \int_0^t\|\nabla \E\|_{L^\infty}\left(\|\mathfrak{U}\|_{L^2}^2+\|\mathfrak{E}\|_{L^2}^2\right)d\tau \non
 \end{split}
\ee
and
\be
\begin{split}
&I_1+I_3+I_4+I_6\\
&\ = \sum_{j=1}^d \int_0^t\int_{\R^d} \Big[ (\mathfrak{E}_j \cdot \nabla \u)\cdot  \mathfrak{E}_j - (\E_j \otimes \mathfrak{E}_j):\nabla \mathfrak{U}+(\hat{\E}_j\cdot \nabla \mathfrak{U})\cdot \mathfrak{E}_j\Big]dxd\tau \\
&\quad\ -\int_0^t\int_{\R^d} \left(\mathfrak{E} \hat{\E}^T\right):\nabla  \mathfrak{U} dxd\tau\\
&\ = \sum_{j=1}^d \int_0^t\int_{\R^d} \Big[ (\mathfrak{E}_j \cdot \nabla \u)\cdot  \mathfrak{E}_j - (\E_j \otimes \mathfrak{E}_j):\nabla \mathfrak{U}\Big] dx d\tau\non\\
&\ \leq C \int_0^t\left(\|\nabla \u\|_{L^\infty}\|\mathfrak{E}\|_{L^2}^2+\| \E\|_{L^\infty}\|\mathfrak{E}\|_{L^2}\|\nabla \mathfrak{U}\|_{L^2} \right)d\tau\non\\
&\ \leq   \frac{\mu}{2} \int_0^t\|\nabla \mathfrak{U}\|_{L^2}^2 d\tau
+ C \int_0^t\left(\|\nabla \u\|_{L^\infty}\|\mathfrak{E}\|_{L^2}^2+ \mu^{-1}\| \E\|_{L^\infty}^2\|\mathfrak{E}\|_{L^2}^2\right)d\tau.\non
\end{split}
\ee
Then we infer from the about estimates and the inequality \eqref{w8} that
\be
\begin{split}
& \|\mathfrak{U}\|_{L^2}^2+\|\mathfrak{E}\|_{L^2}^2+\mu\int_0^t \|\nabla \mathfrak{U}\|^2_{L^2} d\tau
\\
&\ \leq \|\mathfrak{U}_0\|_{L^2}^2+\|\mathfrak{E}_0\|_{L^2}^2 + C\int_0^t h(\tau)\left(\|\mathfrak{U}\|_{L^2}^2+\|\mathfrak{E}\|_{L^2}^2\right)d\tau,\label{gron}
\end{split}
\ee
where
\be
h(t)=\|\nabla \u\|_{L^\infty}+ \|\nabla \E\|_{L^\infty}+\|\E\|_{L^\infty}^2.\non
\ee

We recall that $(\u,\E)$ is assumed to be the global classical solution as in Proposition \ref{glo} with $k\geq 3$, it follows from \eqref{regglo} and the Sobolev embedding theorem $(d=2,3)$ that for $T\in (0, +\infty)$, the solution also satisfies
\begin{equation}\label{55}
\nabla\u\in L^1(0,T; L^\infty(\R^d)), \quad \E\in L^1(0,T; W^{1,\infty}(\R^d))\cap L^2(0,T; L^\infty(\R^d)),
\end{equation}
namely, $h(t)\in L^1(0,T)$.

Since $\mathfrak{U}_0=\mathfrak{F}_0=0$, then by the inequality \eqref{gron} and  the Gronwall lemma, we can conclude that
$$\mathfrak{U}(t)= \mathfrak{E}(t)=0\quad\textrm{for}\quad t\in [0,T].$$

We remark that if $(\u,\E)$ is assumed to be a local classical solution with $H^k$ $(k\geq 3)$ initial data on $[0,T^*]$ for some $T^*>0$ depending on $\|\u_0\|_{H^2}, \|\E_0\|_{H^2}$ as in \cite[Theorem 1]{LLZ08}, then we have $h(t)\in L^1(0, T^*)$, which again indicates the uniqueness result.

The proof of Theorem \ref{unique} is complete.

\section*{Acknowledgements}  Wu was
partially supported by National Science Foundation of China 11371098 and ``Zhuo Xue" Program of Fudan University.


\end{document}